\documentclass[12pt,reqno]{article}

\usepackage{amssymb}
\usepackage{amsmath}
\usepackage{amsthm}
\usepackage{amsfonts}
\usepackage{amscd}
\usepackage{graphicx}
\usepackage{float}
\usepackage{url}
\usepackage[colorlinks=true,
linkcolor=webgreen,
filecolor=webbrown,
citecolor=webgreen]{hyperref}
\usepackage{color}

\definecolor{webgreen}{rgb}{0,.5,0}
\definecolor{webbrown}{rgb}{.6,0,0}

\theoremstyle{plain}

\newtheorem{corollary}{Corollary}
\newtheorem{lemma}{Lemma}

\begin{document}
	
\begin{center}
	\vskip 1cm{\LARGE\bf 
		Collatz Cycles and $\boldsymbol{3n+c}$ Cycles
	}
	\vskip 1cm
	\large
	Darrell Cox\\
	Grayson County College\\
	Denison, TX 75020 \\
	USA\\
	\ \\
	Sourangshu Ghosh\\
	Indian institute of Technology Kharagpur\\
	Kharagpur, West Bengal 721302\\
	India\\
	\ \\
	Eldar Sultanow\\
	Potsdam University\\
	14482 Potsdam\\
	Germany\\
\end{center}

\begin{abstract}	
 \noindent The next element in the $3n+1$ sequence is defined to be $(3n+1)/2$ if $n$ is odd or $n/2$ otherwise. The Collatz conjecture states that no matter what initial value of $n$ is chosen, the sequence always reaches 1 (where it goes into the repeating sequence (1,2,1,2,1,2,\ldots). The only known Collatz cycle is (1,2). Let $c$ be an odd integer not divisible by $3$. Similar cycles exist for the more general $3n+c$ sequence. The $3n+c$ cycles are commonly grouped according to their length and number of odd elements. The smallest odd element in one of these cycles is greater than the smallest odd elements of the other cycles in the group. A \textit{parity vector} corresponding to a cycle consists of 0's for the even elements and 1's for the odd elements. A parity vector generated by the ceiling function is used to determine this smallest odd element. Similarly, the largest odd element in one of these cycles is less than the largest odd elements of the other cycles in the group. A parity vector generated by the floor function is used to determine this largest odd element. This smallest odd element and largest odd element appear to be in the same cycle. This means that the parity vector generated by the floor function can be rotated to match the parity vector generated by the ceiling function. Two linear congruences are involved in this rotation. The natural numbers generated by one of these congruences appear to be uniformly distributed (after sorting). This sequence has properties similar to those of the zeta function zeros. 
\end{abstract}

\newpage
\section{Introduction}
 Halbeisen and Hungerbühler~\cite{hh} found new techniques which allow a refined analysis of rational (and hence integer) Collatz cycles. In particular, they prove optimal estimates for the length of a cycle having positive elements in terms of its minimum. Their main results are reproduced here since they are directly applicable to $3n+c$ cycles. Most lemmas are omitted.

\section{Halbeisen and Hungerbühler's Results for Collatz Cycles}
For $x\in\mathbb{R}$ let $g_{0}(x)=x/2$ and $g_{1}(x)=(3x+1)/2$. Let $\mathbb{Q}[(2)]$ denote the local ring of fractions of $\mathbb{Z}$ at the prime ideal(2), i.e. the domain of all rational numbers having (written in least terms) an odd denominator. A number $p/q\in\mathbb{Q}[(2)]$ with odd $q$ is considered even or odd according to the parity of the numerator $p$. Then the Collatz sequence generated by $x_{0}\in\mathbb{Q}[(2)]$ is defined by $x_{n}=g_{0}(x_{n-1})$ if $x_{n-1}$ is even or $g_{1}(x_{n-1})$ if $x_{n-1}$ is odd for $n\in\mathbb{N}$. Let $S_{l,n}$ denote the set of all 0-1 sequences of length $l$ containing exactly $n$ ones, $S_{l}=\cup^{l}_{n=0}S_{l,n}$ and $S=\cup_{l\in\mathbb{N}}S_{l}$. With every $s=(s_{1},\ldots,s_{l})\in S_{l}$ we associate the affine function $\phi_{s} :\mathbb{R}\rightarrow \mathbb{R}$, $\phi_{s}=g_{s_{l}}\circ\ldots\circ g_{s_{2}} \circ g_{s_{1}}$. A sequence
$(x_{0},\ldots,x_{l})$ of real numbers $x_{i}$ is called a \textit{pseudo-cycle of length $l$} if there exists $s=(s_{1},\ldots,s_{l})\in S_{l}$ such that (1) $\phi_{s}(x_{0})=x_{0}\in\mathbb{Q}[(2)]$ and (2) $g_{s_{i+1}}(x_{i})=x_{i+1}$ for $i=0,\ldots,l-1$. \\

\noindent Notice that if $p/q\in \mathbb{Q}$ with $2^{r}|q$ then $2^{r}|\tilde{q}$ where $\tilde{q}$ denotes the denominator of $g_{i}(p/q)$ $(i=0,1)$. Hence every element of a pseudo-cycle is in $\mathbb{Q}[(2)]$. Thus, if $p/q$ and $g_{i}(p/q)=\tilde{p}/\tilde{q}$ are consecutive elements of a pseudo-cycle, then $i=0$ if $p$ is even (since else $\tilde{p}/\tilde{q} \notin\mathbb{Q}[(2)]$) or $i=1$ if $p$ is odd (since else $\tilde{p}/\tilde{q} \notin\mathbb{Q}[(2)]$). The conclusion of this observation is given by the following lemma~\ref{Lemma:1}.

\begin{lemma}
The set of pseudo-cycles coincides with the set of Collatz cycles in 
$\mathbb{Q}[(2)]$. Cycles consist of either positive or negative elements.
\label{Lemma:1}
\end{lemma}

\noindent The function $\varphi : S\rightarrow \mathbb{N}$ will be defined recursively by $\varphi(\O)=0$, $\varphi(s_{0})=\varphi(s)$, and $\varphi(s_{1})=3\varphi(s)+2^{l(s)}$ where $s$ denotes an arbitrary element of $S$ and $l(s)$ the length of $s$. The function $\varphi$ is computed explicitly by $\varphi(s)=\sum_{j=1}^{l(s)}s_{j}3^{s_{j+1}+\ldots+s_{l(s)}}2^{j-1}$. \\

\noindent A consequence of the above definition is the decomposition formula $\varphi(s\bar{s})=3^{n(\bar{s})}\varphi(s)+2^{l(s)}\varphi(\bar{s})$. Here $s\bar{s}$ is the concatenation of $s$, $\bar{s}\in S$, and $n(s)$ denotes the number of 1's in the sequence $s$. The next lemma~\ref{lemma:2} shows how $\varphi$ is used to explicitly compute the function $\phi_{s}$.

\begin{lemma}
(Lagarias~\cite{la}). For arbitrary $s\in S$, $\phi_{s}(x)=\frac{3^{n(s)}x+\varphi(s)}{2^{l(s)}}$
and hence for every $s\in S$ there exists a unique $x_{0}\in \mathbb{Q}[(2)]$ which generates a Collatz cycle in $\mathbb{Q}[(2)]$ of length $l(s)$ and which coincides with the pseudo-cycle generated by $s$. The value $x_{0}$ is given by $x_{0}=\frac{\varphi(s)}{2^{l(s)}-3^{n(s)}}$.
\label{lemma:2}
\end{lemma}

\begin{proof}
The proof is by induction with respect to $l(s)$. (1) $l(s)=1$: This is checked from the definition. (2) $l(s)>1$: If $s=\bar{s}_{0}$ then $\phi_{\bar{s}0}(x)=\frac{\phi_{\bar{s}}(x)}{2}=\frac{3^{n(\bar{s})}x+\varphi(\bar{s})}{2\cdot 2^{l(\bar{s})}}=\frac{3^{n(s)}x+\varphi(s)}{2^{l(s)}}$. The case $s=\bar{s}1$ is analogous.
\end{proof}
 
\noindent For $s\in S_{l}$ let $\sigma(s)$ denote the orbit of $s$ in $S_{l}$ generated by the left-shift permutation $\lambda_{l}$ : $(s_{1},\ldots,s_{l})\rightarrow (s_{2},..,s_{l},s_{1})$, i.e. $\sigma(s):= \{\lambda_{l}^{k}(s)$ : $k=1,\ldots,l\}$. Furthermore, let $M_{l,n}$ denote $\max_{s\in S_{l,n}}$$\{{\min_{t\in \sigma(s)} \varphi(t)}\}$. \\

\noindent Now suppose the Collatz conjecture is verified for all initial values $x_{0}\le m$. If one can then show that $\forall n,l<L: \frac{M_{l,n}}{2^{l}-3^{n}} \le m$, it follows that the length of a Collatz cycle in $\mathbb{N}$ which does not contain 1 is at least $L$. \\

\noindent Let $\tilde{s}$ denote the sequence for which $\varphi$ attains the value $M_{l,n}$.

\begin{lemma}
Let $n\le l$ be natural numbers. Let $\tilde{s}_{i}:=\lceil in/l\rceil - \lceil(i-1)n/l\rceil$ (for $1\le i\le l$). Then $\varphi(\bar{s})=\min_{t\in \sigma(\bar{s})}\{\varphi(t)\}=M_{l,n}$.
\end{lemma}

\begin{corollary}
For every $l$ and $n\le l$ we have
$M_{l,n}=\sum_{j=1}^{l}(\lceil jn/l\rceil- \lceil(j-1)n/l\rceil)2^{j-1}3^{n-\lceil jn/l\rceil}$.
\end{corollary}

\section {\texorpdfstring{The Minimum Element in a $\boldsymbol{3n+c}$ Cycle}{The Minimum Element in a 3n+c Cycle}}
Setting $c$ to $2^{l(s)}-3^{n(s)}$ in Lemma~\ref{lemma:2} gives integer $3n+c$ cycles. A staircase for $\tilde{s}$ where $27$ and $n=17$ along with a staircase representing the partial sums of $\lfloor in/l \rfloor - \lfloor (i-1)n/l \rfloor$ is given in Figure~\ref{fig:1}.

% trim=left bottom right top
\begin{figure}[H]
	\includegraphics[clip, trim=0cm 3.8cm 3cm 0cm, width=\linewidth]{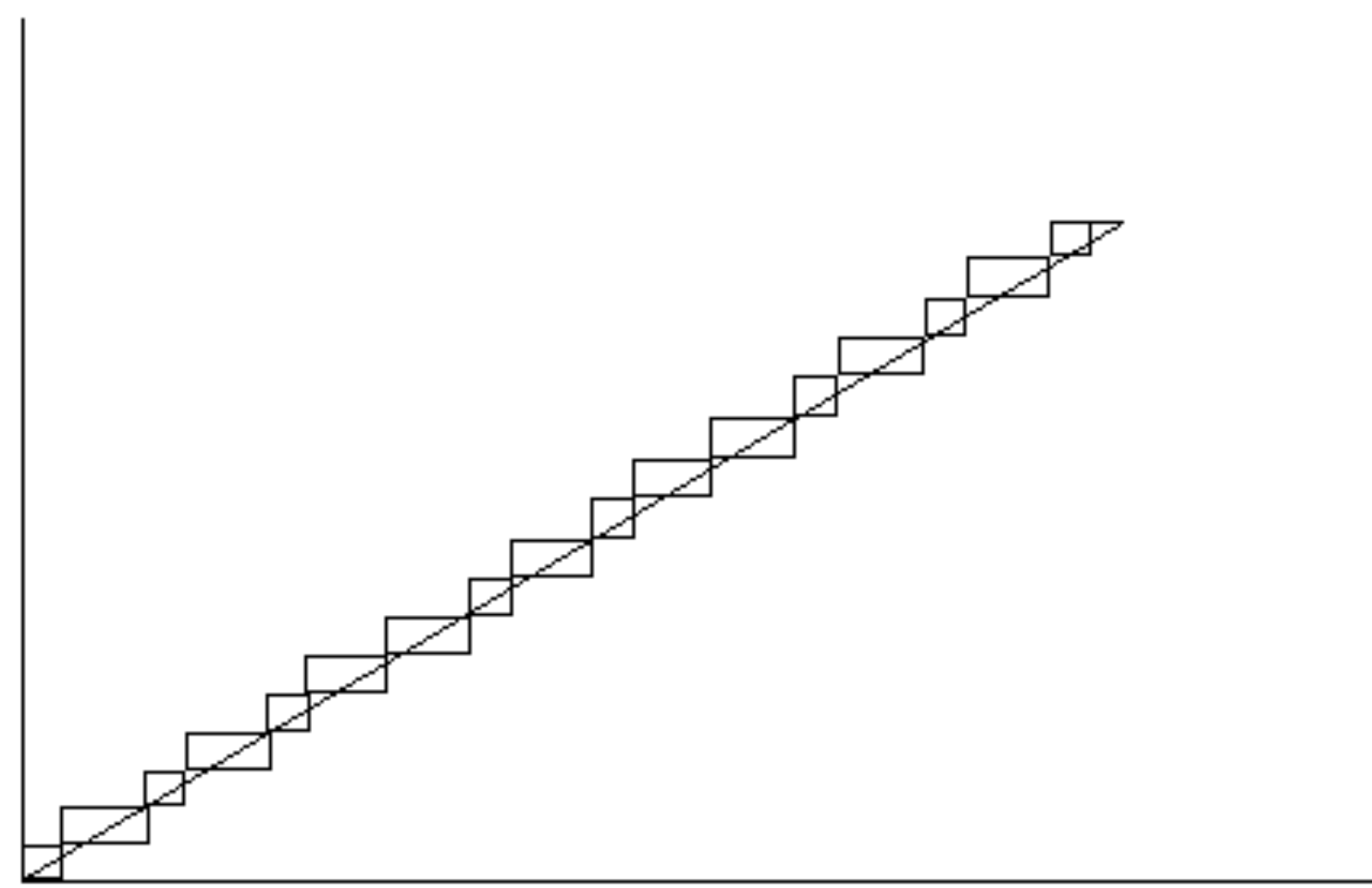}
	\caption{staircase representing the partial sums of $\lfloor in/l\rfloor-\lfloor(i-1)n/l\rfloor$}
	\label{fig:1}
\end{figure} 
 
\noindent The staircase using the floor function can be viewed as being an upside-down staircase where Halbeisen and Hungerbühler's logic can be used to find a lower bound of the maximum odd element in a $3n+c$ cycle. Let $t_{j}=\lceil jn/l\rceil - \lceil (j-1)n/l \rceil$, $j=1,\ldots,l$. This parity vector is an element of $S_{l,n}$. Let $r$ denote $\gcd(l,n)$. The parity vector $\lfloor jn/l \rfloor - \lfloor(j-1)n/l \rfloor$, $j=1,\ldots,l$, consists of $r$ identical sub-vectors. Similarly, the parity vector $t_{j}$ consists of $r$ identical sub-vectors and each of these sub-vectors is the same as the corresponding sub-vector of $\lfloor jn/l \rfloor - \lfloor (j-1)n/l \rfloor$, $j=1,\ldots,l$, except for the first and last elements. First suppose that $l$ and $n$ are relatively prime. When the parity vector $\lfloor jn/l\rfloor - \lfloor(j-1)n/l\rfloor$, $j=1,\ldots,l$ is right-rotated by one position (corresponding to a multiplication by 2), it matches $t_{j}$ except for the first two elements of each sub-vector. The first mismatch corresponds to a loss of $3^{n-1}$ and the second mismatch corresponds to a gain of $2\cdot 3^{n-1}$. In general, the loss is $\sum_{i=0}^{r-1}2^{i(l/r)}3^{n-1-i(n/r)}$. Let $N_{l,n}$ denote $2M_{l,n}-\sum_{i=0}^{r-1}2^{i(l/r)}3^{n-1-i(n/r)}$. A primitive $3n+c$ cycle doesn't have any common divisors of its elements. A generalization of Halbeisen and Hungerbühler's result is given by Corollary~\ref{co:generalization_hh}:

\begin{corollary}
If $c=2^{l}-3^{n}$, $M_{l,n}$ is greater than or equal to the minimum elements in the $3n+c$ cycles corresponding to $s\in S_{l,n}$ (not necessarily primitive) and $N_{l,n}$ is less than or equal to the maximum odd elements in the cycles.
\label{co:generalization_hh}
\end{corollary}

\noindent The elements of the $3n+c$ cycles are $\varphi(t)_{t \in\sigma(s)}$ where $s\in S_{l,n}$. From the definition of $N_{l,n}$, it is not apparent that it is in a cycle, but it appears to be in the same cycle as $M_{l,n}$. \\

\noindent For example, for $(l,n)=(6,4)$, the parity vector for $M_{l,n}$ is $(1, 1, 0, 1, 1, 0)$, $c=-17$, and the odd elements of the cycle containing $M_{l,n}$ and $N_{l,n}$ are $(85, 119,85,119)$. There is only one more element in $S_{l,n}$ and its odd elements are $(65,89,125,179)$. The smallest odd element in the cycle containing $M_{l,n}$ (85) is greater than 65 and the largest odd elements in the cycle containing $N_{l,n}$ (119) is less than 125. The cycle with odd elements of $(85,119)$ is not primitive and reduces to a cycle with odd elements of $(5,7)$ for $c=-1$. When $l$ and $n$ are not relatively prime and $c=2^{l}-3^{n}$, the cycles generated from $M_{l,n}$ are not primitive. This is due to the duplicated sub-vectors in the parity vector forming a geometric progression. This geometric progression is the same as in the expansion of $(a^{x}-b^{x})/(a-b)$. Reducing the cycle generated from $M_{6,4}$ effectively divides $2^{6}-3^{4}$ by $(2^{6}-3^{4})/(2^{3}-3^{2})$. \\

\noindent When $l=11$ and $n=7$, $M_{l,n}=3767$, $N_{l,m}=6805$, and $2^{l}-3^{n}=-139$. The quotient $3767/139$ (approximately equal to 27) is greater than the minimum element in the $c=-1$ cycle $(34,17,25,37,55,82,41,61,91,136,68)$ and $6805/139$ (approximately equal to $49$) is less than the maximum odd element. For the $c=-1$ cycle of $(5,7,10)$, $M_{3,2}=5$ and $N_{3,2}=7$ ($-1$ equals $2^{3}-3^{2}$). For $c=-17$, the cycles are $(85,119,170,85,119,170)$, $(103,146,73,101,143,206)$, and $(65,89,125,179,260,130)$ ($2^{6}-3^{4}=-17$). The first cycle contains $M_{6,4}$ (equal to $85$) and $N_{6,4}$ (equal to $119$). As expected, $85$ is greater than $73$ and $65$ and $119$ is less than $143$ and $179$. The cycle $(85,119,170)$ is not primitive and reduces to the $c=-1$ cycle. For the $c=1$ cycle of $(4,1)$, $2^{2}-3^{1}=1$ and $M_{2,1}=N_{2,1}=1$. There can be no other such $c=1$ cycles due to the Catalan conjecture (proved by Mih\v{a}ilescu~\cite{m}). This theorem states that the only natural number solutions of $x^{a}-y^{b}$ are $x=3$, $a=2$, $y=2$, and $b=3$. This leaves the possibility of $3n+c$ cycles where $s\in S_{l,n}$ that are not primitive and reduce to $c=1$ cycles. \\

\noindent All the parity vectors in $S$ are used up by the $3n+c$ cycles where $c=2^{l}-3^{n}$. Two $3n+c$ cycles with different $c$ values can't have the same parity vector. For example, the elements of a $c=5$ cycle are $(19,31,49,76,38)$ and a $c=7$ sequence having the same parity vector is $(65,101,155,236,118,\ldots)$. The ratios of the odd elements are 0.2923, 0.3069, and 0.3161 and would have to keep increasing to match the iterations of the $3n+5$ cycle. So the unreduced $3n+c$ cycles where $c=2^{l}-3^{n}$ account for all possible primitive $3n+c$ cycles.

\section{Statistical Results}

Let $d$ denote the rotation of the floor parity vector required to match the ceiling parity vector (measured in the clockwise direction). This quantity appears to satisfy the congruence $n\cdot d\equiv -1\bmod l$ when $\gcd(l,n)=1$. A similar congruence is $d(n-x)-(l-d)x\equiv -1\bmod l$. This congruence was derived using the staircases and can be solved given the $d$ value so that the values of $(n-x,x)$ are not specific to properties of $3n+c$ cycles. \\

\noindent For a real number $x$, let $[x]$ denote the integral part of $x$ and $\{x\}$ the fractional part. The sequence $\omega=\{x_{n}\}$, $n=1,2,3,\ldots$ of real numbers is said to be uniformly distributed modulo $1$ (abbreviated u.d. $\bmod 1$) if for every pair $a, b$ of real numbers with $0\le a<b\le 1$ we have $\lim_{N\rightarrow \infty}\frac{A([a,b);N;\omega)}{N}=b-a$. The formal definition of u.d. mod 1 was given by Weyl~\cite{we}~\cite{we1}. Let $\Delta: 0=z_{0}<z_{1}<z_{2}<\ldots$ be a subdivision of the interval $[0, \infty)$ with $\lim_{k\rightarrow\infty}z_{k}=\infty$. For $z_{k-1}\le x<z_{k}$ put $[x]_{\Delta}=z_{k-1}$ and $\{x\}_{\Delta}=\frac{x-z_{k-1}}{z_{k}-z_{k-1}}$ so that $0\le\{x\}_{\Delta}<1$. The sequence of $(x_{n})$, $n=1,2,3,\ldots$ of non-negative real numbers is said to be uniformly distributed modulo $\Delta$ (abbreviated u.d. mod $\Delta$) if the sequence $(\{x_{n}\}_\Delta)$, $n=1,2,3,\ldots$ is u.d. $\bmod 1$. The notion of u.d. $\bmod\Delta$ was introduced by Leveque~\cite{lev}. \\

\noindent If $f$ is a function having a Riemann integral in the interval $[a,b]$, then its integral is the limit of Riemann sums taken by sampling the function $f$ in a set of points chosen from a fine partition of the interval. This is then a criterion for determining if a sequence is uniformly distributed. A sequence of real numbers is uniformly distributed (mod 1) if and only if for every Riemann-integrable function $f$ on $[0,1]$ one has $\lim_{N\rightarrow\infty}1/N\sum_{n\le N}f(\{x_{n}\})=\int_{0}^{1}f(x)dx$. In the following, evidence that $n-x$ and $x$ are u. d. mod $\Delta$ is presented using this criterion and Weyl's criterion~\cite{we}~\cite{we1}. \\

\noindent Weyl's criterion is that $(\gamma_{n})$ is u.d. mod 1 if and only if $\lim_{N\rightarrow\infty}1/N\sum_{n=1}^{N}e^{2\pi im\gamma_{n}}=0$ for every integer $m\ne 0$. In the following, the $z$ increments in the subdivision are set to $\sqrt{2}$ to avoid any aliasing with the integer $n-x$ and $x$ values. A plot of the resulting sequence generated from the sorted $n-x$ values for $n=1,\ldots,l-1$ and $l=997$ is given in Figure~\ref{fig:2}.

\begin{figure}[H]
	\includegraphics[clip, trim=0cm 1cm 0cm 0cm, width=\linewidth]{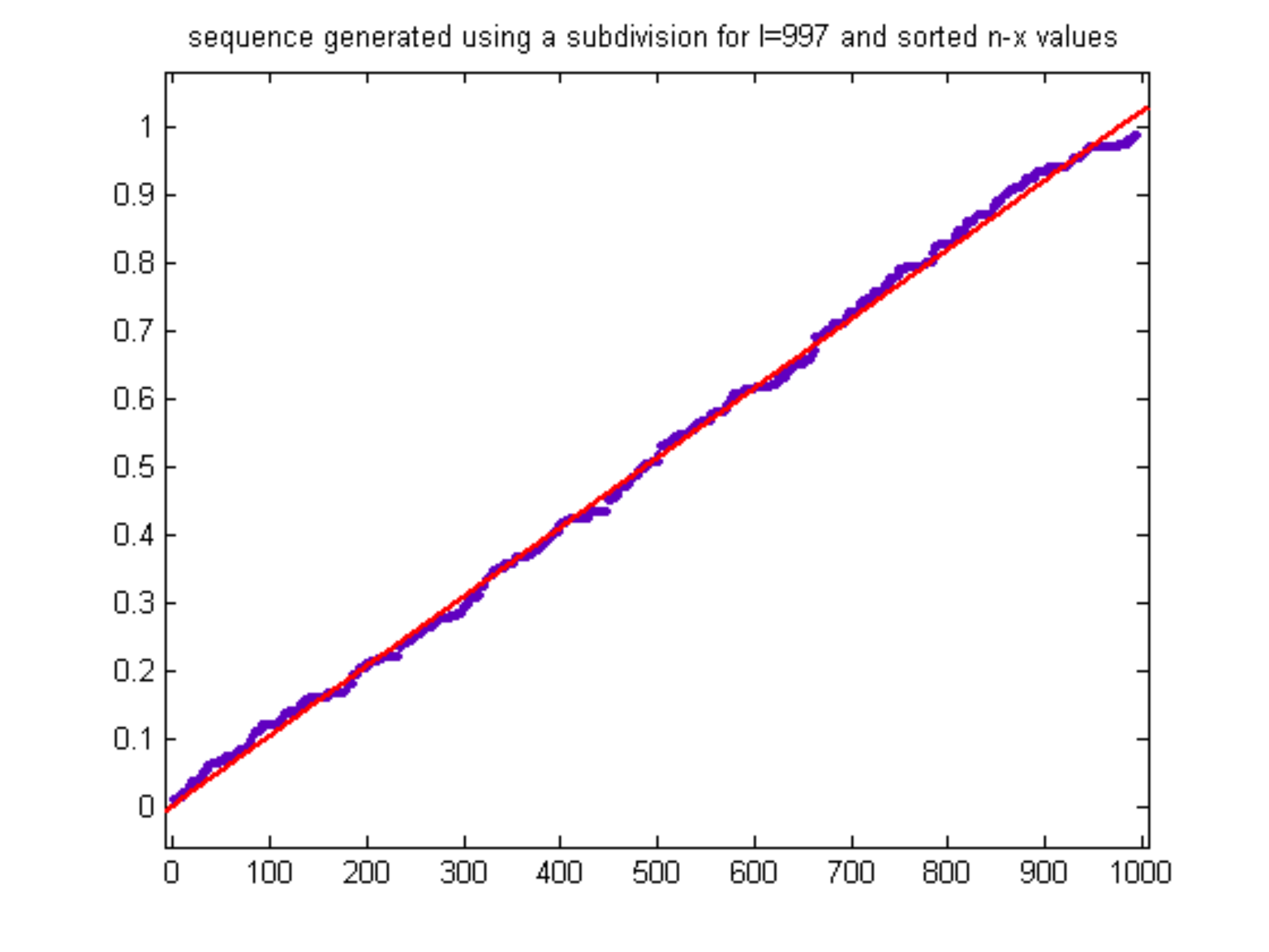}
	\caption{sequence generated from the sorted $n-x$ values for $n=1,\ldots,l-1$ and $l=997$}
	\label{fig:2}
\end{figure}

\noindent In the following, $(\gamma_{n})$ is set to such sequences and the moduli of the complex-valued results are computed. The moduli for $l=1999$, $n-x$, and $m=1$ are given in Figure~\ref{fig:3}.

\begin{figure}[H]
	\includegraphics[clip, trim=0cm 1cm 0cm 0cm, width=\linewidth]{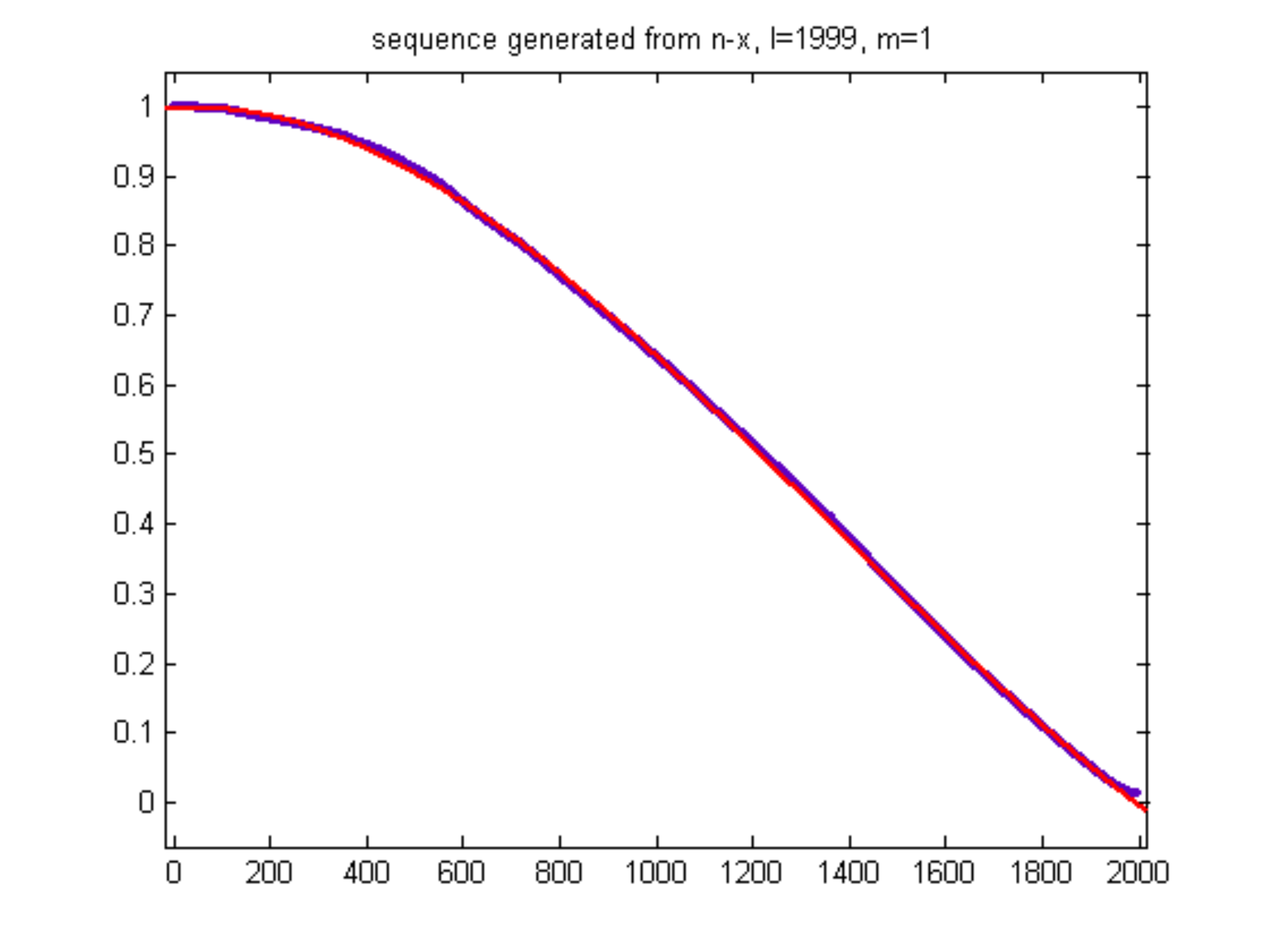}
	\caption{Moduli for $l=1999$, $n-x$, and $m=1$}
	\label{fig:3}
\end{figure}

\noindent A cubic least-squares fit of the curve (where R-squared=0.9999) is included. The moduli for $l=1999$, $x$, and $m=1$ (excluding 16 values of zero in the input sequence) are given in Figure~\ref{fig:4}.

\begin{figure}[H]
	\includegraphics[clip, trim=0cm 1cm 0cm 0cm, width=\linewidth]{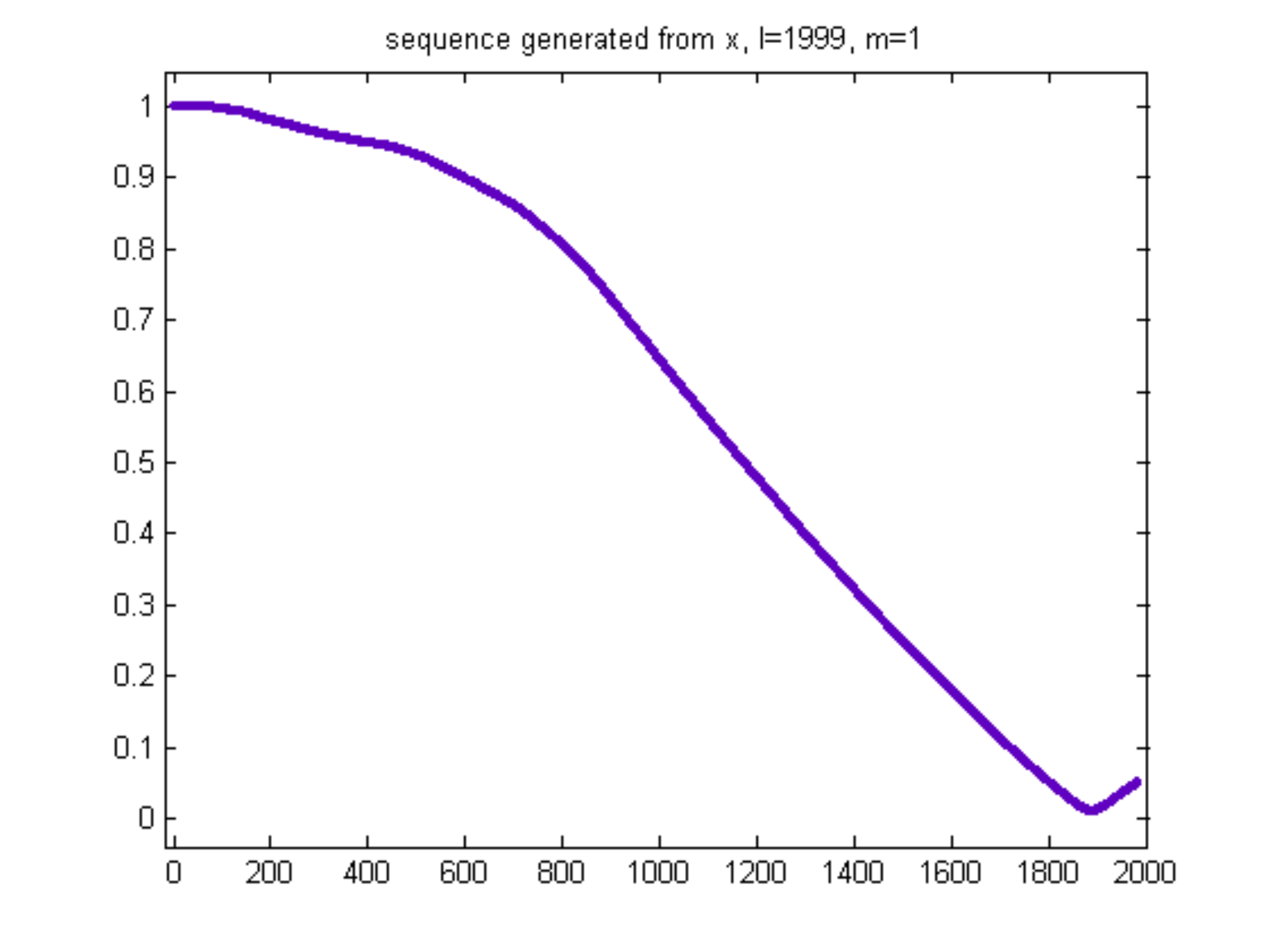}
	\caption{Moduli for $l=1999$, $x$, and $m=1$}
	\label{fig:4}
\end{figure}

\noindent The moduli for $l=9973$, $n-x$, and $m=10$ are given in Figure~\ref{fig:5}.

\begin{figure}[H]
	\includegraphics[clip, trim=0cm 1cm 0cm 0cm, width=\linewidth]{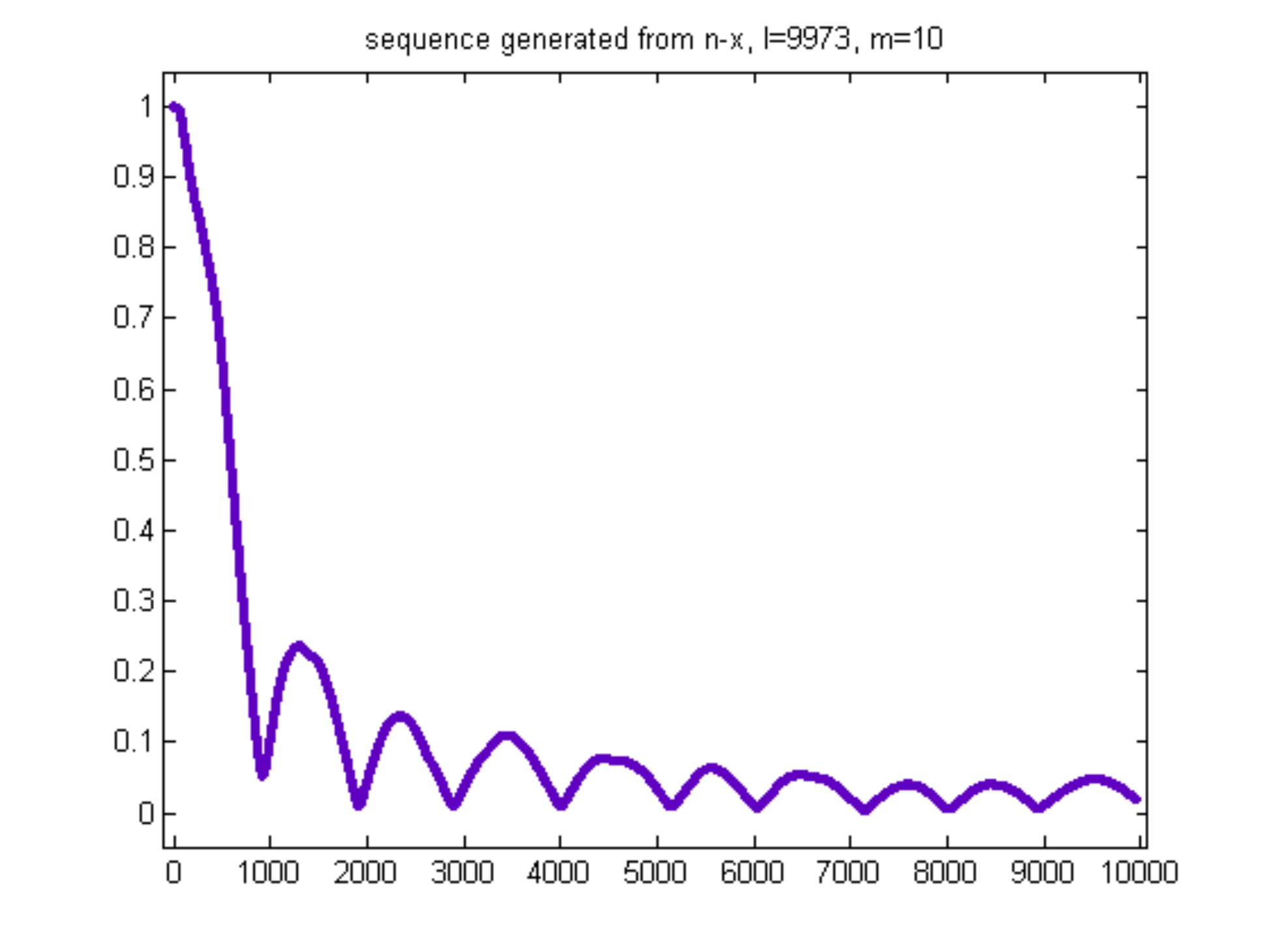}
	\caption{Moduli for $l=1999$, $x$, and $m=1$}
	\label{fig:5}
\end{figure}

\noindent The moduli for $l=9973$, $x$, and $m=10$ (excluding 18 values of zero in the input sequence) are given in Figure~\ref{fig:6}.

\begin{figure}[H]
	\includegraphics[clip, trim=0cm 1cm 0cm 0cm, width=\linewidth]{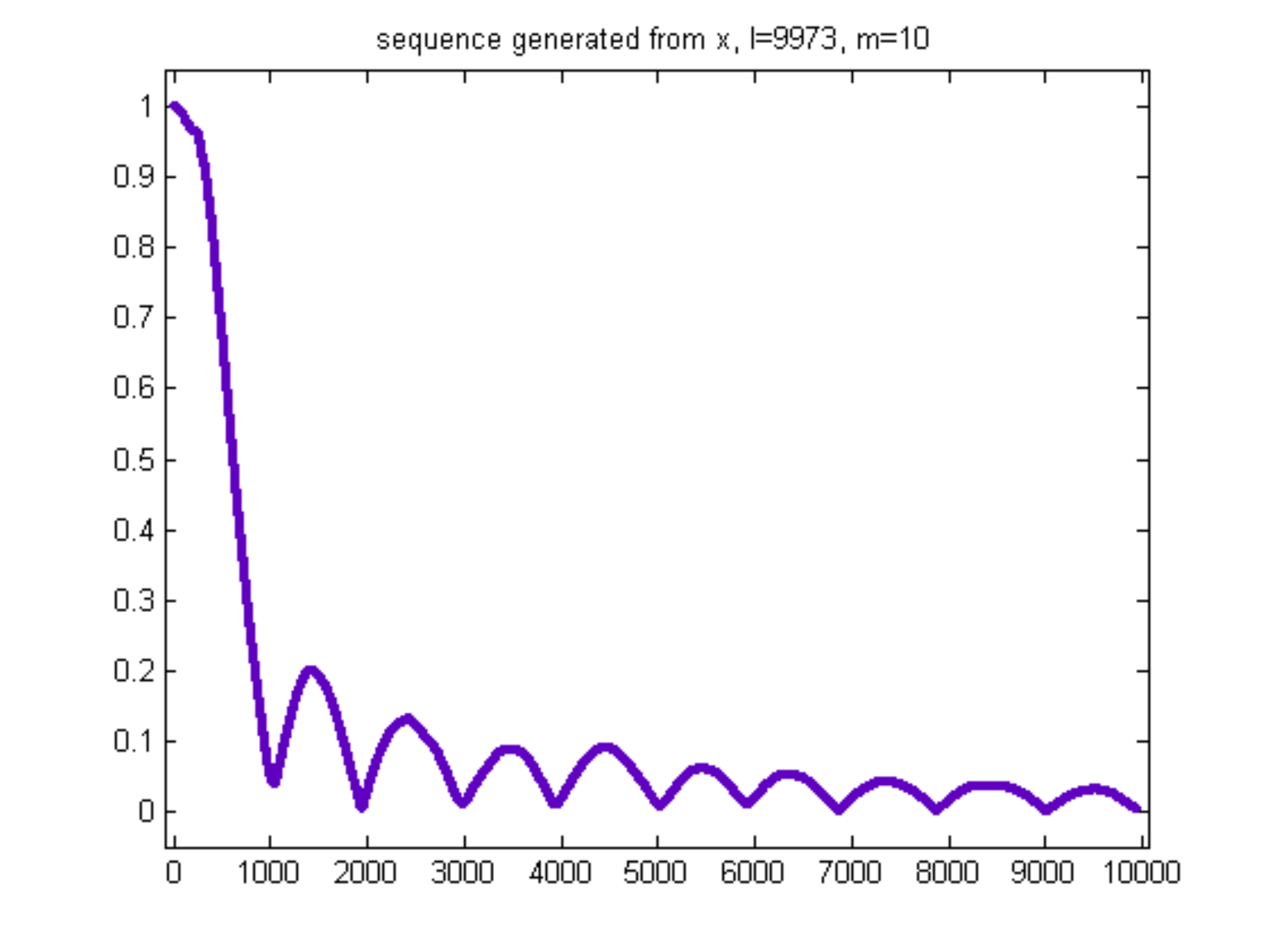}
	\caption{Moduli for $l=9973$, $x$, and $m=10$}
	\label{fig:6}
\end{figure}

\noindent In general, there are $m-1$ oscillations in such curves. \\

\noindent The functions $f(x)$ to be considered are $x$, $x^{2}$, $x^{3}$, $x^{4}$, $\sqrt{x}$, $\sqrt{\sqrt{x}}$, $\log(x)$, $e^{x}$, $\sin(x)$, $\cos(x)$, $\tan(x)$, and $\frac{1}{a^{2}+x^{2}}$. The values of $\int_{0}^{1}f(x)dx$ are $1/2$, $1/3$, $1/4$, $1/5$, $2/3$, $4/5$, $-1$, $2.72$, $0.84$, $0.54$, $1.56$, and $\frac{1}{a}\tan^{-1}{\frac{x}{a}}$ (equal to $0.23$ for $a=2$ and $0.11$ for $a=3$) respectively. For $l=997$ and the sequence generated from $x$, the results are $0.49$, $0.31$, $0.22$, $0.17$, $0.66$, $0.80$, $-1.00$, $2.69$, $0.84$, $0.56$, $1.45$, $0.23$ (for $a=2$), and $0.11$ (for $a=3$) respectively. For $l=9973$ and the sequence generated from $n-x$, the results are $0.50$, $0.34$, $0.25$, $0.20$, $0.67$, $0.80$, $-0.97$, $2.74$, $0.84$, $0.54$, $1.55$, $0.23$ (for $a=2$), and $0.11$ (for $a=3$) respectively. \\

\noindent The trigonometric functions require a fixed amount to be added to the sequence values (apparently to change the phase). The exponential function also requires a fixed amount to be added to the sequence values - the same as for the cosine function. Apparently, this is due to Euler's formula $e^{ix}=\cos(x)+i\cdot \sin(x)$. Denote the amounts for sine and cosine by $j$ and $k$ respectively. These values satisfy the equation $j^{2}+k^{2}=\cos(1)$, similar to the formula $\sin(x)^{2}+\cos(x)^{2}=1$. They also satisfy the equation $j/k=\sqrt{\tan(1)}$, similar to the formula $\sin(x)/\cos(x)=\tan(x)$. The amount required for the sine function is $\sin^{-1}(\cos(1))$. The amount required for the cosine function can be determined by using the formula $j^{2}+k^{2}=\cos(1)$. The amount required for the tangent function is $1/e$ ($e=\tan(1)/j$). \\

\noindent See Cox and Ghosh~\cite{cg} for more graphs.

\section{Discrete Uniform Distributions and the M\"{o}bius Function}

\noindent Cox~\cite{cox} investigated convolving the zeta function zeros with the M\"{o}bius function. This method is applicable to any uniformly distributed sequence. In the following, the $n-x$ values are ordered in increasing value. A plot of $n-x$ convolved with the M\"{o}bius function for $l=1999$ (a prime) is given in Figure~\ref{fig:7}.

\begin{figure}[H]
	\includegraphics[clip, trim=0cm 1cm 0cm 0cm, width=\linewidth]{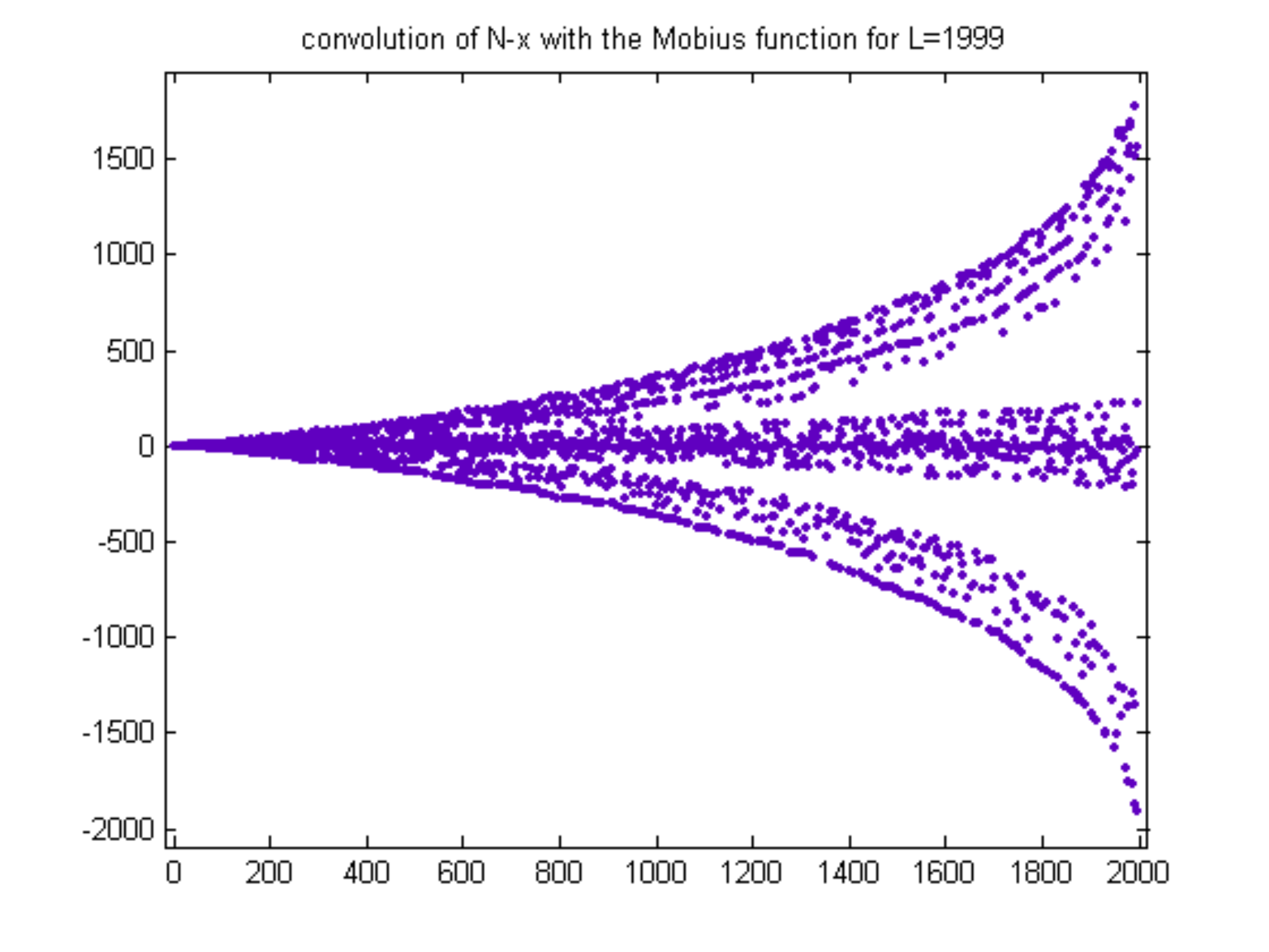}
	\caption{Convolution of $n-x$ with the M\"{o}bius function}
	\label{fig:7}
\end{figure}

\noindent The convolution consists of many curves. The bottom curve corresponds to the 302 primes less than 1999. Let $p$ and $q$ denote distinct primes. A plot of the curves at $pq$ locations is given in Figure~\ref{fig:8}.

\begin{figure}[H]
	\includegraphics[clip, trim=0cm 1cm 0cm 0cm, width=\linewidth]{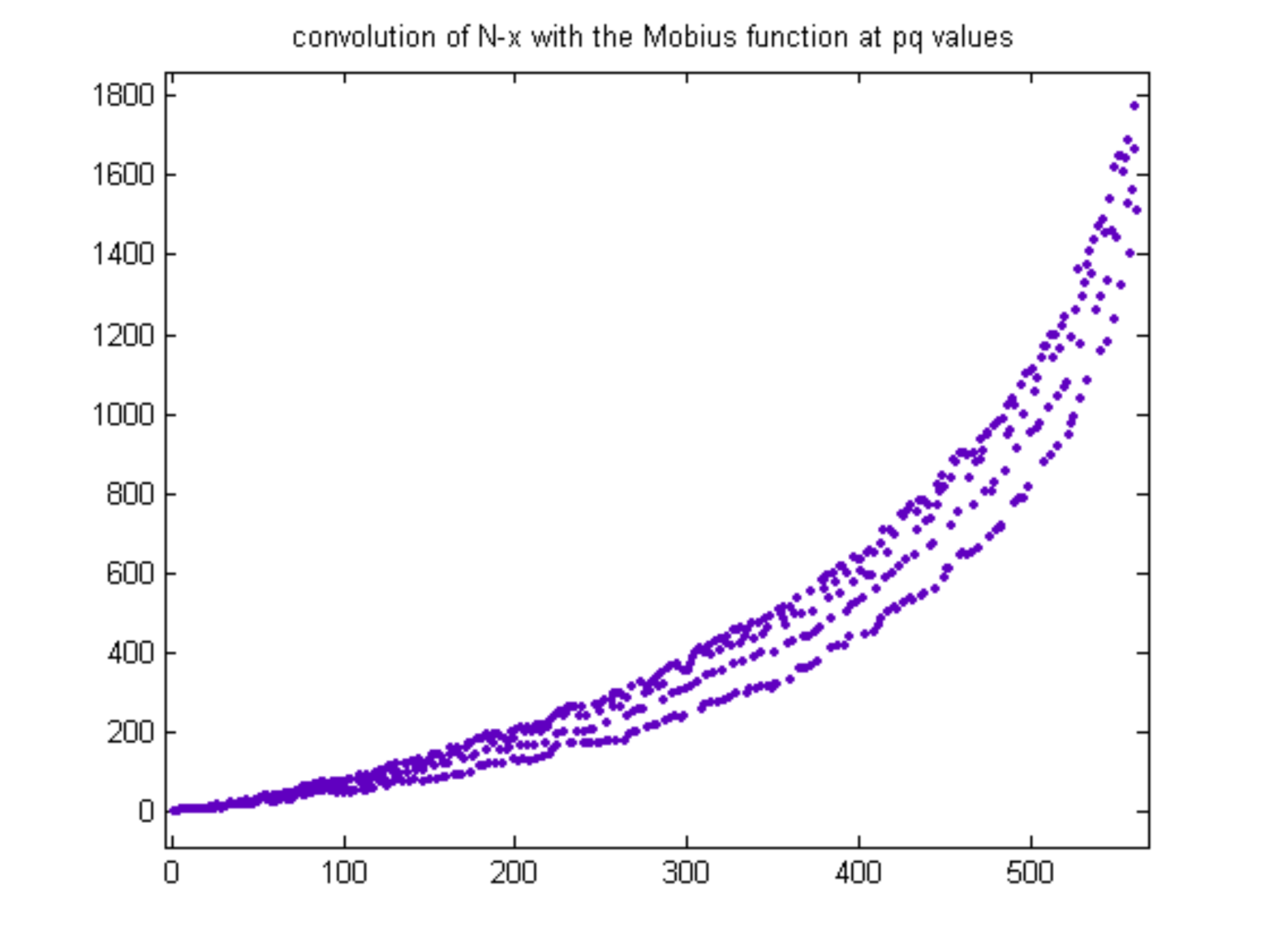}
	\caption{Curves at $pq$ locations}
	\label{fig:8}
\end{figure} 

\noindent The convolution of the differences between the adjacent $n-x$ values with the M\"{o}bius function for a particular curve is normally distributed. For the above $pq$ curve, the mean is 0.5879 with a 95\% confidence interval of (0.3742, 0.8016) and the standard deviation is 2.5812 with a 95\% confidence interval of (2.4387, 2.7415). A plot of this distribution along with the corresponding probit function is given in Figure~\ref{fig:9}.

\begin{figure}[H]
	\includegraphics[clip, trim=0cm 1cm 0cm 0cm, width=\linewidth]{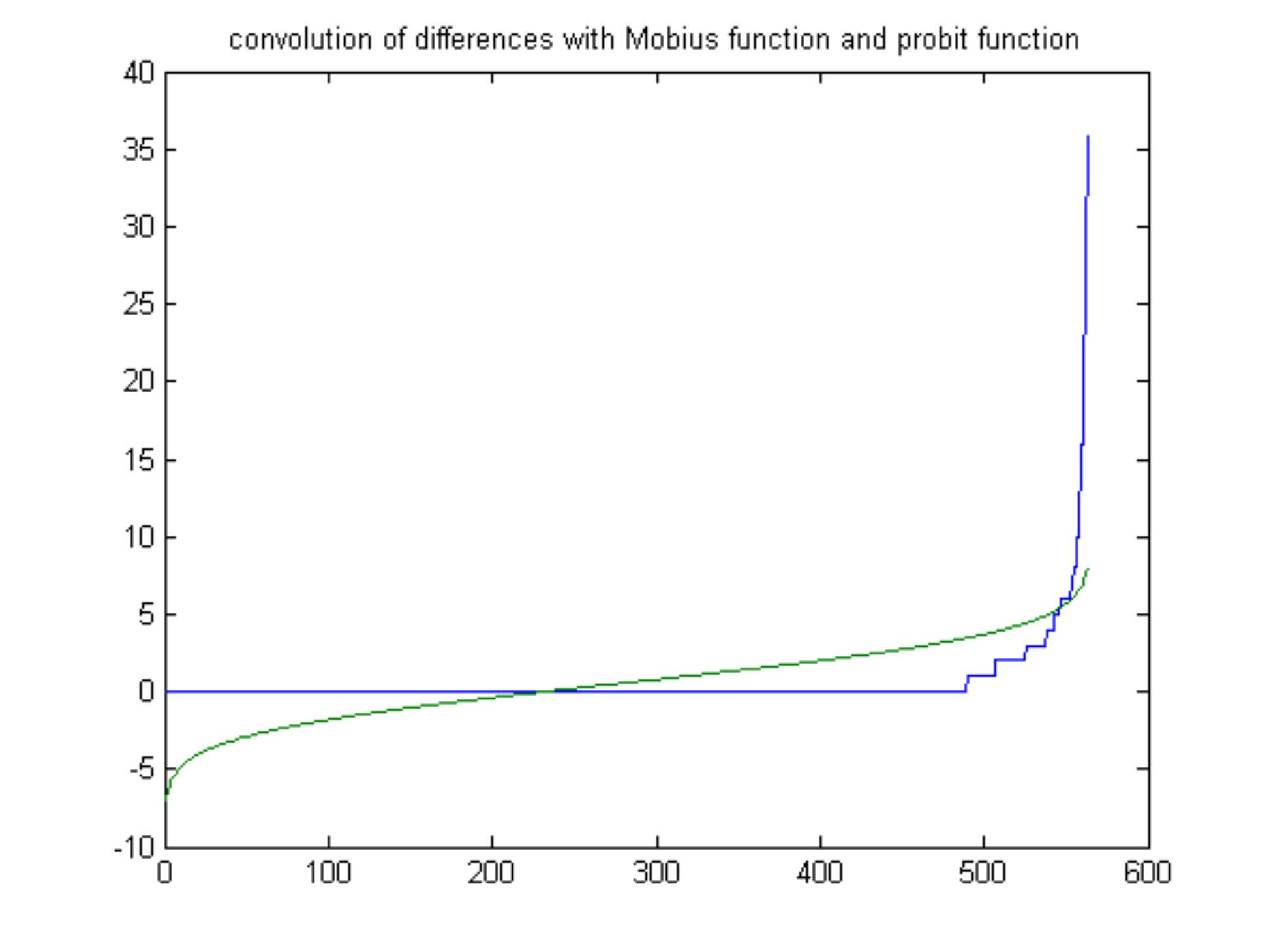}
	\caption{Convolution of differences with Möbius function}
	\label{fig:9}
\end{figure}

\noindent The probit function is the inverse cumulative distribution function of the standard normal distribution. The more general function $F^{-1}(p)=\mu+\sigma\Phi^{-1}(p)$ where $\mu$ and $\sigma$ are the mean and standard deviation of the normal distribution is used here. The poor fit is partially due to the discrete values of the distribution. \\

\noindent A plot of the zeta function zeros convolved with the M\"{o}bius function for $l=1,2,3,\ldots,2000$ is given in Figure~\ref{fig:10}.

\begin{figure}[H]
	\includegraphics[clip, trim=0cm 1cm 0cm 0cm, width=\linewidth]{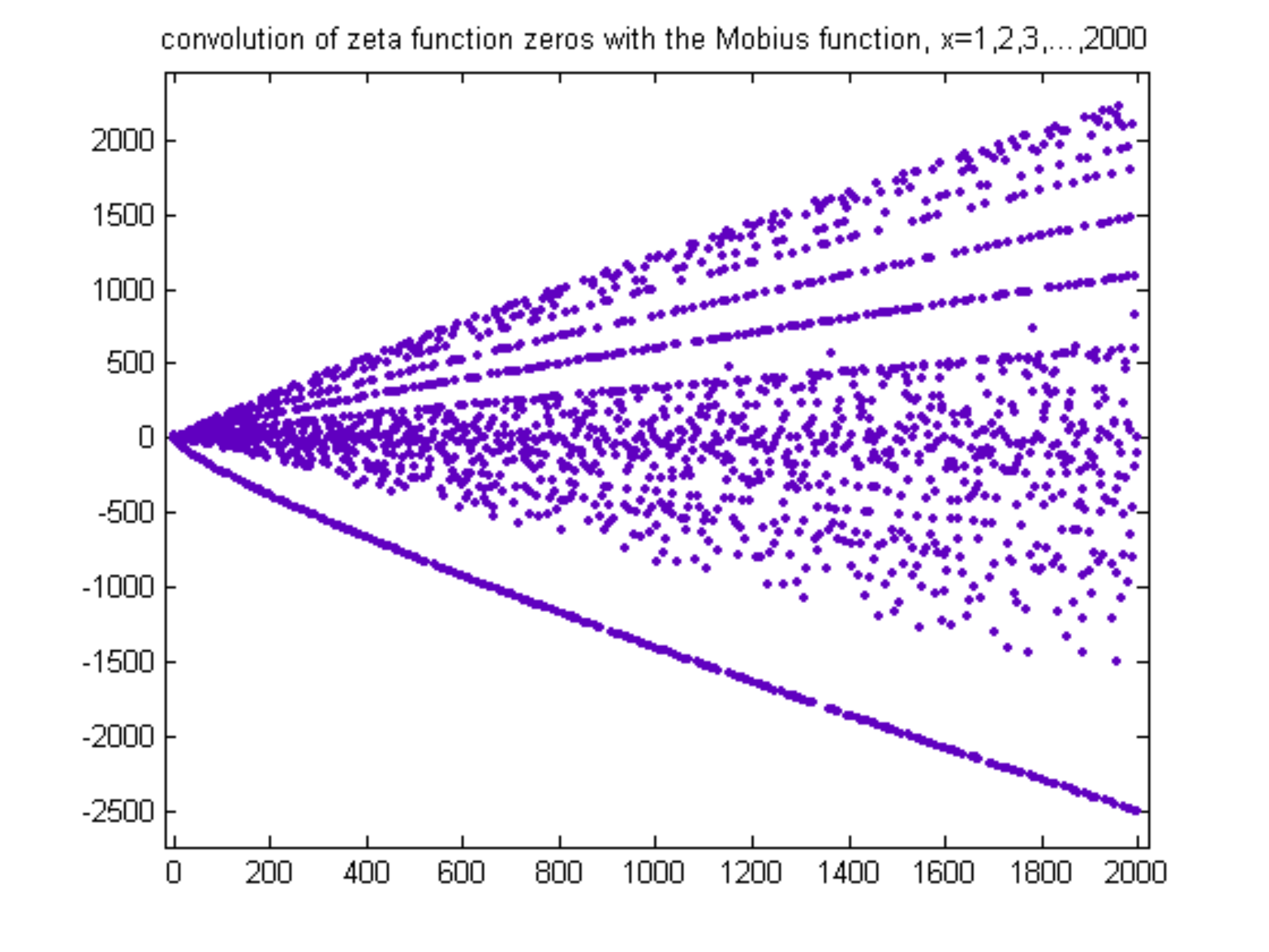}
	\caption{Convolution of zeta function zeros with M\"{o}bius function}
	\label{fig:10}
\end{figure}

\noindent The bottom curve corresponds to the 303 primes less than 2000. A plot of the curves at $pq$ locations is given in Figure~\ref{fig:11}.

\begin{figure}[H]
	\includegraphics[clip, trim=0cm 1cm 0cm 0cm, width=\linewidth]{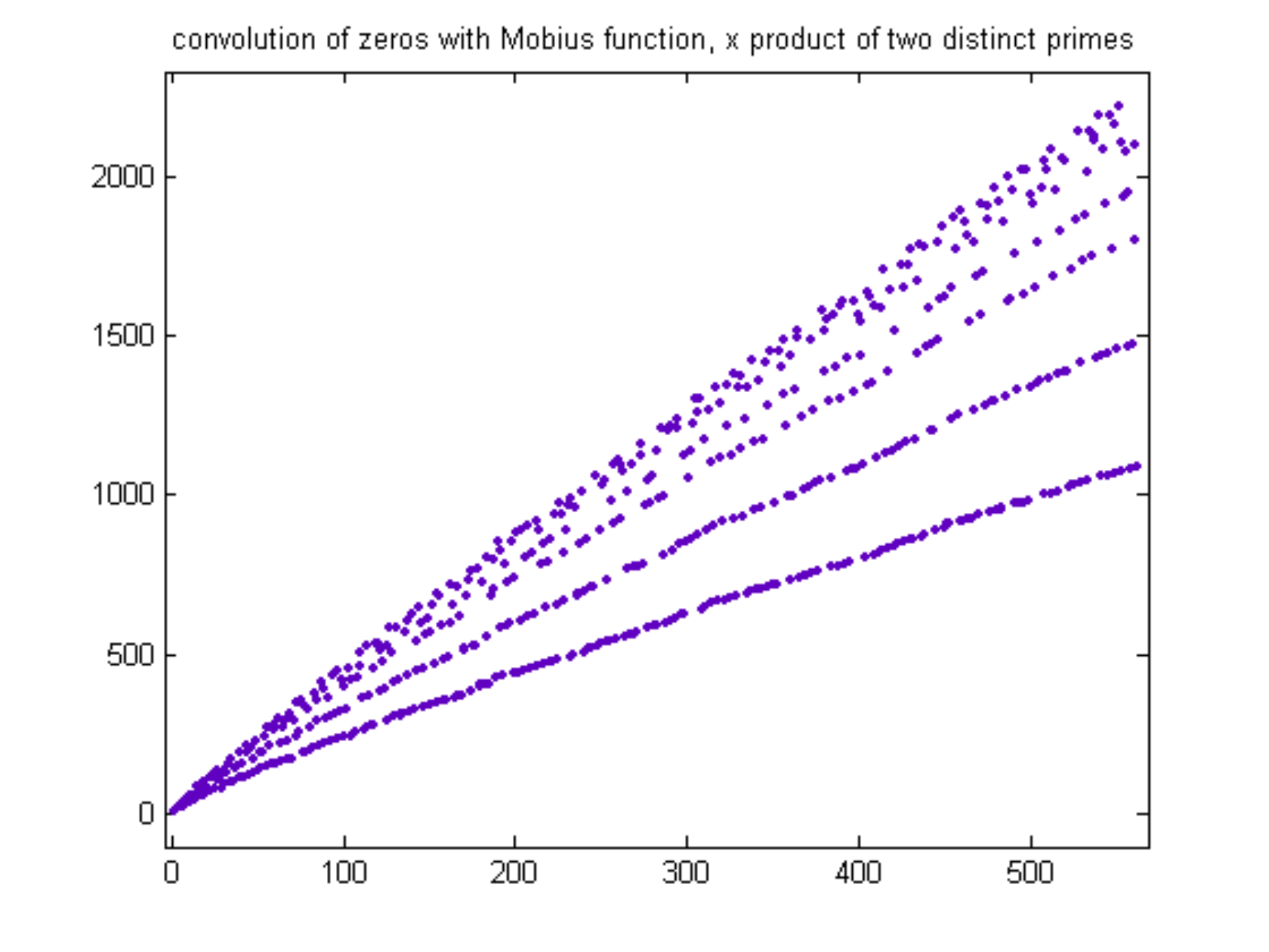}
	\caption{Curves at pq locations}
	\label{fig:11}
\end{figure}

\noindent The convolution of the differences between the adjacent zeta function zeros values with the M\"{o}bius function for a particular curve is normally distributed. For the above $pq$ curve, the mean is 2.7126 with a 95\% confidence interval of (2.6015, 2.8237) and the standard deviation is 1.3419 with a 95\% confidence interval of (1.2678, 1.4252). A plot of this distribution along with the corresponding probit function is given in Figure~\ref{fig:12}.

\begin{figure}[H]
	\includegraphics[clip, trim=0cm 1cm 0cm 0cm, width=\linewidth]{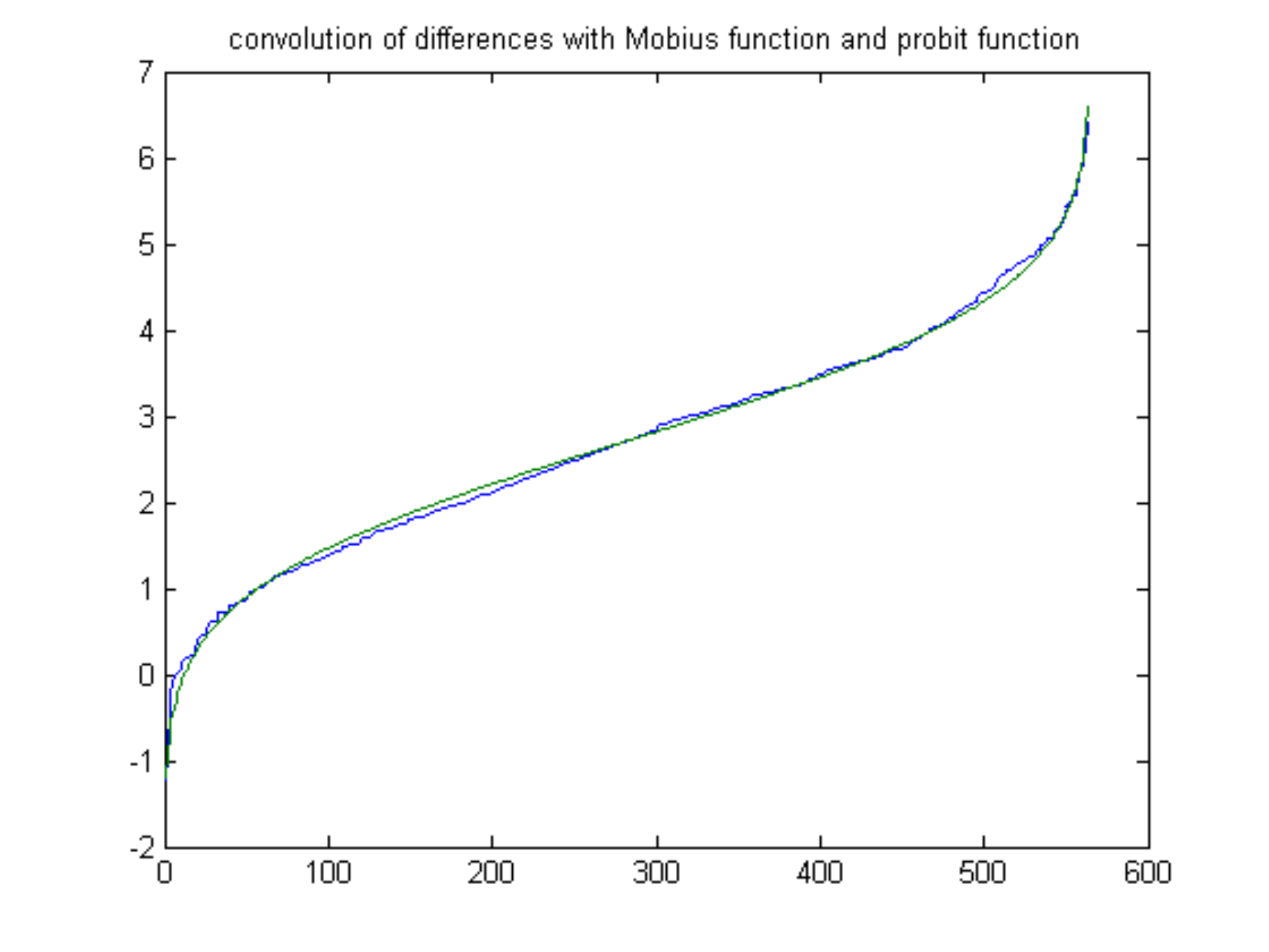}
	\caption{Convolution of differences with M\"{o}bius function}
	\label{fig:12}
\end{figure}

\noindent The results for the zeta function zeros are similar to those for the $n-x$ values.

\section{More Statistical Results}

\noindent In this section, $l$ is restricted to being prime. Since $l$ cannot divide $nd+1$ and $n'd+1$ for equal $d$, the mapping of the rotated floor parity vector to the ceiling parity vector for all possible $n$ values is one-to-one. The values of $nd+1$ modulo $l$ determines a "basis". For example, the basis for $l=31$ is \\

$
\begin{array}{rrl}
&     & n~\text{values} \\        
~\text{least residue}=&  1: & 30,15,10,6,5,3,2,1 \\
&  3: & 23,4              \\ 
&  5: & 22,14,11,7         \\
&  7: & 27,24,18,12,9,8    \\
&  8: & 19,13              \\
& 11: & 20,17              \\
& 15: & 29,16              \\
& 19: & 28,21              \\
& 21: & 26,25              \\
&     &
\end{array}
$

\noindent The number of elements in this basis is 9. The number of distinct prime factors in the respective elements is $\{3,2,3,2,2,3,2,3,3\}$. The maximum number of distinct prime factors is 3. \\

\noindent A plot of the number of elements in a basis versus the primes less than 10000 is given in Figure~\ref{fig:13}.

\begin{figure}[H]
	\includegraphics[clip, trim=0cm 1cm 0cm 0cm, width=\linewidth]{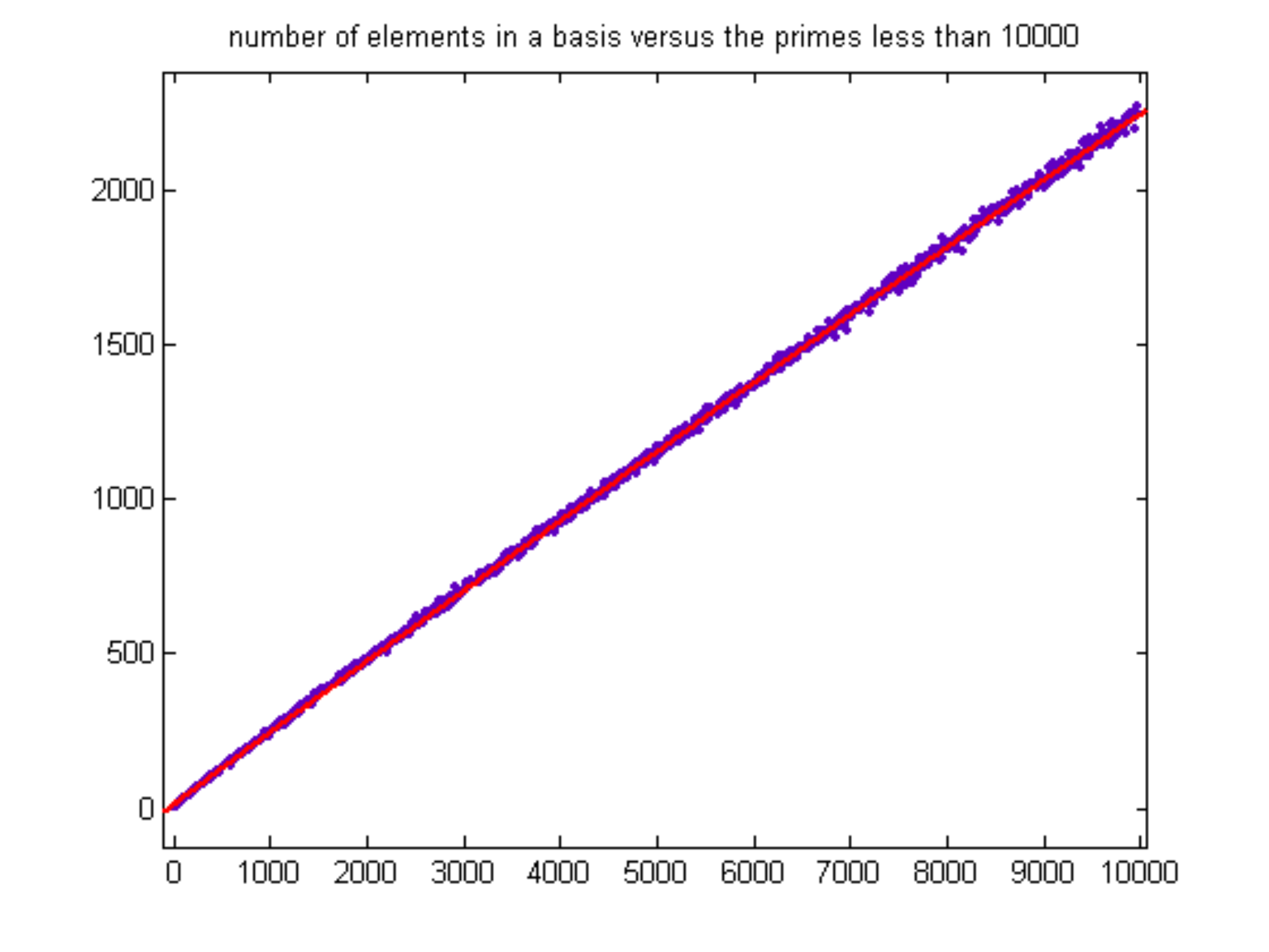}
	\caption{Number of elements in a basis versus the primes}
	\label{fig:13}
\end{figure}

\noindent For a quadratic least-squares fit of the curve, $p_{1}=-9.91\cdot 10^{-7}$ with a 95\% confidence interval of ($-1.074\cdot 10^{-7}$,$-9.076\cdot 10^{-7}$), $p_{2}=0.2332$ with a 95\% confidence interval of (0.2324, 0.2341), $p_{3}=12.96$ with a 95\% confidence interval of (11.26, 14.66), SSE=$1.596\cdot 10^{5}$, R-squared=0.9997, and RMSE=11.41. \\

\noindent A plot of the maximum number of distinct prime factors of the elements of a basis versus the square roots of the primes less than 10000 is given in Figure~\ref{fig:14}.

\begin{figure}[H]
	\includegraphics[clip, trim=0cm 1cm 0cm 0cm, width=\linewidth]{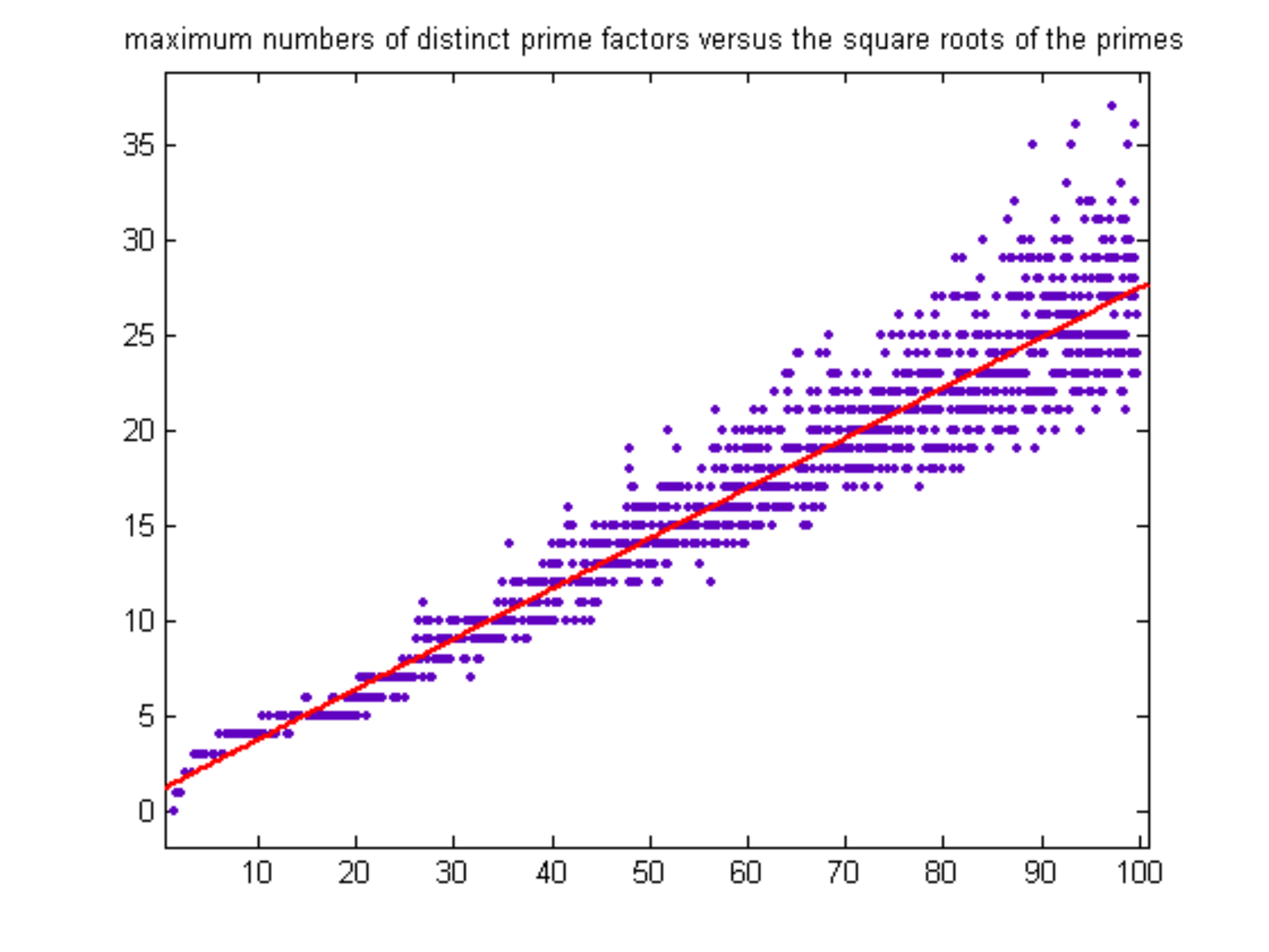}
	\caption{Maximum number of distinct prime factors versus square roots of primes}
	\label{fig:14}
\end{figure}

\noindent For a linear least-squares fit of the curve, $p_{1}=0.2636$ with a 95\% confidence interval of (0.2588, 0.2685), $p_{2}=1.123$ with a 95\% confidence interval of (0.7893, 1.456), SSE=5875, R-squared=0.9017, and RMSE=2.189. \\

\noindent A plot of the logarithm of the histogram of the number of elements in a basis is given in Figure~\ref{fig:15}.

\begin{figure}[H]
	\includegraphics[clip, trim=0cm 1cm 0cm 0cm, width=\linewidth]{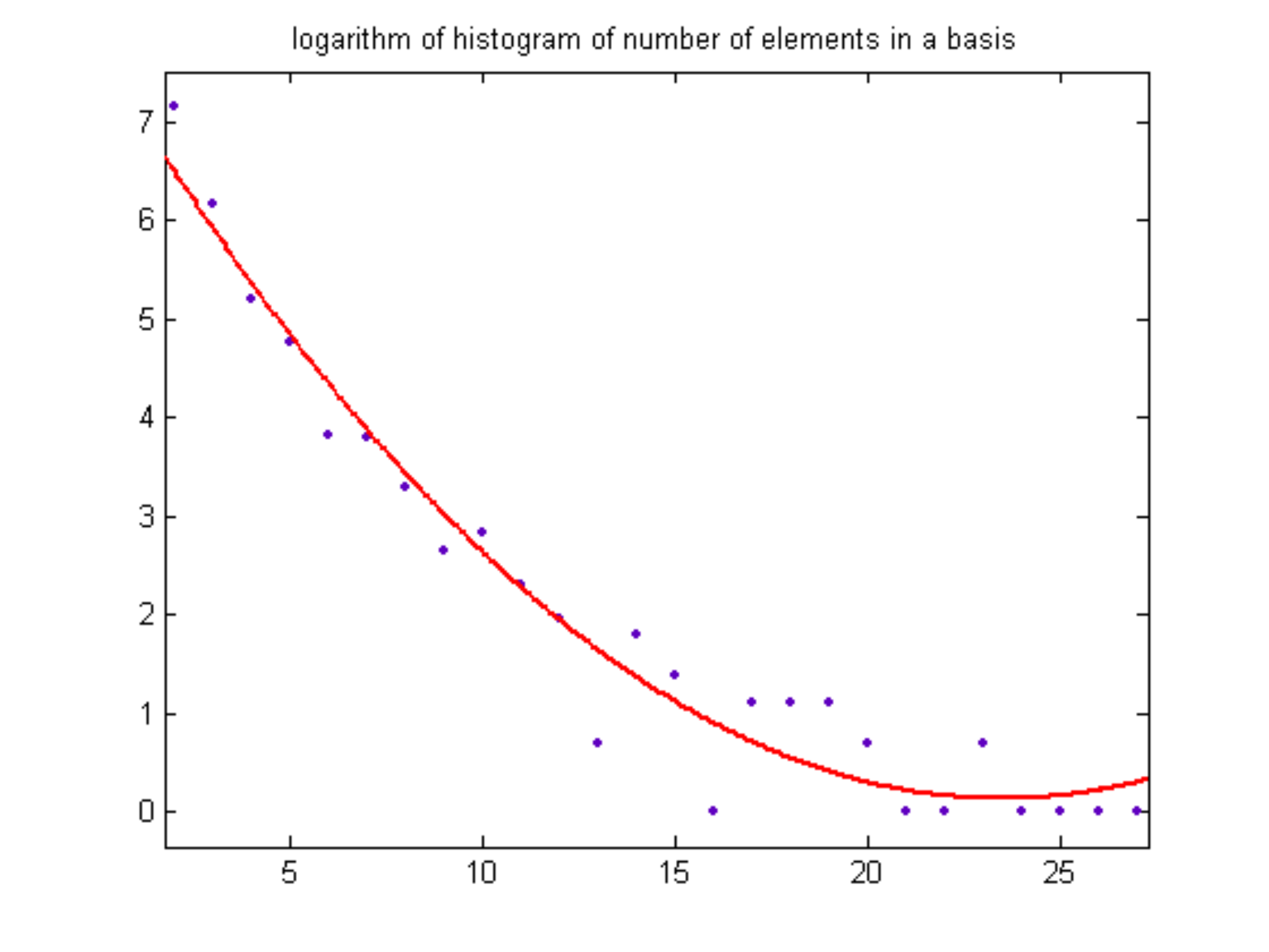}
	\caption{Logarithm of histogram of number of elements in a basis}
	\label{fig:15}
\end{figure}

\noindent For a quadratic least-squares fit of the curve, $p_{1}=0.01388$ with a 95\% confidence interval of (0.01025, 0.01751), $p_{2}=-0.6509$ with a 95\% confidence interval of (-0.759, -0.5428), $p_{3}=7.759$ with a 95\% confidence of (7.7073, 8.445), SSE=4.648, R-squared=0.9568, and RMSE=0.9496.


\newpage
\begin{thebibliography}{99}

\bibitem{hh} L.~Halbeisen and N.~Hungerbühler, Optimal bounds for the length of rational Collatz cycles, {\em Acta Arith.}, \textbf{LXXVIII.3} (1997), 227--239

\bibitem{la} J.~C.~Lagarias, The set of rational cycles for the $3x+1$ problem, {\em Acta Arith.}, \textbf{56} (1990), 33--53

\bibitem{m} P.~Mih\v{a}ilescu, Primary Cyclotomic Units and a Proof of Catalan's Conjecture, {\em J. reine angew. Math.} \textbf{572} (2004), 167--195

\bibitem{we} H.~Weyl, \"{U}ber ein Problem aus dem Gebiete der diophantischen Approximationen, {\em Nachr. Ges. Wiss. Göttingen}, Math.-phys. Kl., \textbf{1914}, 234--244

\bibitem{we1} H.~Weyl, \"{U}ber die Gleichverteilung von Zahlen mod. Eins, {\em Math. Ann.}, \textbf{77} (1916), 313--352

\bibitem{lev} W.~J.~LeVeque, On uniform distribution modulo a subdivision, {\em Pacific J. Math.}, \textbf{3} (1953), 757--771

\bibitem{cg} D.~Cox and S.~Ghosh, A Uniformly Distributed Congruence, 10.13140/RG.2.2.19763.35363 (2021)

\bibitem{cox} D.~Cox, Zeta Function Zeros, the Möbius Function, and Dirichlet Products, 10.13140/RG.2.2.14588.97923 (2019)

\end{thebibliography}
\end{document}